\documentclass[11pt,a4paper]{article}

\usepackage{amsmath,amsthm,amssymb,a4wide}
\usepackage{times}
\usepackage{enumerate}

\usepackage{setspace}

\usepackage{dsfont}

\newcommand{\R}{\mathbb R}

\newcommand{\gep}{\varepsilon}

\theoremstyle{definition}
\newtheorem{thm}{Theorem}[section]
\newtheorem{cor}[thm]{Corollary}
\newtheorem{prop}[thm]{Proposition}

\newtheorem{defin}[thm]{Definition}
\newtheorem{rem}[thm]{Remark}

\newtheorem{exa}[thm]{Example}

\newcommand{\subjclass}[1]{\bigskip\noindent\emph{2010 Mathematics Subject Classification:}\enspace#1}

\numberwithin{equation}{section}

\begin{document}

\title{The Aronsson equation, Lyapunov functions and local Lipschitz regularity of the minimum time function}
\author{Pierpaolo Soravia\thanks{email: soravia@math.unipd.it.  }\\
Dipartimento di Matematica\\ Universit\`{a} di Padova, via Trieste 63, 35121 Padova, Italy}

\date{}
\maketitle

\begin{abstract}
We define and study $C^1-$solutions of the Aronsson equation (AE), a second order quasi linear equation. We show that such super/subsolutions make the Hamiltonian monotone on the trajectories of the closed loop Hamiltonian dynamics. We give a short, general proof that $C^1-$solutions are absolutely minimizing functions.
We discuss how $C^1-$supersolutions of (AE) become special Lyapunov functions of symmetric control systems, and allow to find continuous feedbacks driving the system to a target in finite time, except on a singular manifold.
A consequence is a simple proof that the corresponding minimum time function is locally Lipschitz continuous away from the singular manifold, despite classical results show that it should only be  H\"older continuous unless appropriate conditions hold.
We provide two examples for H\"ormander and Grushin families of vector fields where we construct $C^1-$solutions (even classical) explicitly.
\end{abstract}

\subjclass{Primary 49L20; Secondary 35F21, 35D40, 93B05.}

\section{Introduction}

In this note we want to describe a possible new, non standard way of using the Aronsson equation, a second order partial differential equation, to obtain controllability properties of deterministic control systems. We investigate a symmetric control system
\begin{equation}\label{eqsystem}
\left\{
\begin{array}{ll}
\dot x_t=f(x_t,a_t),\\
x_0=x_o\in\Omega,
\end{array}\right.
\end{equation}
where $-f(a,A)\subset f(x,A)$, $A$ is a nonempty and compact subset of a metric space. We define the Hamiltonian
$$H(x,p)=\max_{a\in A}\{-f(x,a)\cdot p\},$$
which is therefore nonnegative and positively one homogeneous in the adjoint variable, and we want to drive the system to a target, temporarily we say the origin.
We are interested in the relationship of (\ref{eqsystem}) with the Aronsson equation (AE)
$$-\nabla\left(H(x,\nabla U(x))\right)\cdot H_p(x,\nabla U(x))=0,$$
which is a quasilinear degenerate elliptic equation.
Ideally, if everything is smooth, when we are given a classical solution $U$ of (AE) and we consider a trajectory $x_t$ of the Hamiltonian dynamics
$$\dot x_t=-H_p(x_t,\nabla U(x_t)),$$
which is a closed loop dynamics for the original control system, we find out that (AE) can be rewritten as
$$\frac d{dt}H(x_t,\nabla U(x_t))=0.$$
Therefore $H(x_t,\nabla U(x_t))$ is constant. This is a very desirable propery on the control system since it allows to use $U$ as a control Lyapunov function, despite the presence of a possibly nonempty singular set
\begin{equation}\label{eqsing}{\mathcal H}=\{x:H(x,\nabla U(x))=0\},
\end{equation}
which possibly contains the origin.
Indeed if $x_o$ is outside the singular set and $U$ has a unique global minimum at the origin, then the trajectory of the Hamiltonian dynamics will reach the origin in finite time.

In general however, several steps of this path break down. From one side, (AE) does not have $C^2$ classical solutions in general. Even in the case where $f=a$, $A=B_1(0)\subset\R^n$ is the closed unit ball, $H(p)=|p|$ and (AE) becomes the well known infinity Laplace equation
$$-\Delta U(x)\;\nabla U(x)\cdot \nabla U(x)=0,$$
solutions are not classical, although known regularity results show that they are $C^{1,\alpha}$. Therefore solutions of (AE) have to be meant in some weak sense, as viscosity solutions. For generic viscosity solutions, we can find counterexamples to the fact that the Hamiltonian is constant along trajectories of the Hamiltonian dynamics, as we show later. For an introduction to the theory of viscosity solutions in optimal control, we refer the reader to the book by Bardi, Capuzzo-Dolcetta \cite{bcd}.

In this paper we will first characterize when, for a given super or subsolution of (AE) the Hamiltonian is monotone on the trajectories of the Hamiltonian dynamics (e.g. satisfies the {\it monotonicity property}). To this end we introduce the notion of $C^1-$super/subsolution and prove for them that they satisfy the monitonicity property of the Hamiltonian. We emphasize the fact that not all viscosity solution that are $C^1$ functions, are $C^1-$ solutions according to our definition. Moreover, as a side result, we also show that our $C^1-$solutions are absolutely minimizing functions, i.e. local minimizers of the functional that computes the $L^\infty$ norm of the Hamiltonian. It is a well know equivalent property to being a viscosity solution of (AE) at least when $H$ is coercive or possibly in some Carnot Caratheodory spaces, but this fact is not completely understood in general. Therefore $C^1-$solution appears to be an appropriate notion.

We then prove that if (AE) admits a $C^1-$supersolution $U$ having a unique minimum at the origin, then our control system can be driven to the origin in finite time with a continuous feedback, starting at every initial point outside the singular set $\mathcal H$. If moreover $U$ satisfies appropriate decay in a neighborhood of the origin only at points where the Hamiltonian $H$ stays away from zero, then we show that the corresponding minimun time function is locally Lipschitz continuous outside the singular set, despite the fact that even if the origin is small time locally attainable, then the minimum time function can only be proved to be H\"older continuous in its domain, in general, under appropriate conditions. Thus the loss of regularity of the minimum time function is only concentrated at points in the singular set. Finally for two explicit well known examples, where the system has an H\"ormander type, or a Grushin family of vector fields, we exibit two explicit not yet known classical solutions of (AE), their gauge functions, providing examples of smooth absolute minimizers for such systems and the proof that their minimum time function is locally Lipschitz continuous outside the singular set.
We remark the fact that neither in the general statement nor in the examples, the family of vector fields is ever supposed to span the whole space at the origin, therefore the classical sufficient attainability condition ensuring that the minimum time function is locally Lipschitz continuous will not be satisfied in general. Indeed in the explicit examples that we illustrate in Section 4, the minimum time function is known to be locally only $1/2-$H\"older continuous in its domain.

Small time local attainability and regularity of the minimum time function is an important subject in optimal control.
Classical results by Petrov \cite{pe} show sufficient conditions for attainability at a single point by requiring that the convex hull of the vector fields at the point contains the origin in its interior. Such result was later improved by Liverovskii \cite{li} augmenting the vector fields with the family of their Lie brackets, see also the paper by author \cite{so7}.
More recently such results had several extensions in the work by Krastanov and Quincampoix \cite{kr3} and Marigonda, Rigo and Le \cite{ma,ma3, ma4}.
Our regularity results rather go in the direction of those contained in two recent papers by Albano, Cannarsa and Scarinci \cite{alcasc, alcasc1}, where they show, by completely different methods, that if a family of smooth vector fields satisfies the H\"ormander condition, then the set where the local Lipschitz continuity of the minimum time function fails is the union of singular trajectories, and that it is analytic except on a subset of null measure. Our approach is instead more direct and comes as a consequence of constructing Lyapunov functions as $C^1-$supersolutions of the Aronsson equation. We finally mention the paper by Motta and Rampazzo \cite{mr} where the authors study higher order hamiltonians obtained by adding iterated Lie brackets as additional vector fields, in order to prove global asymptotic controllability to a target. While we do not study asymptotic controllability in this paper, their idea of constructing a higher order Hamiltonian may be seen complementary to ours, using instead the equation (AE).

Equation (AE) was introduced by Aronsson \cite{ar0}, as the Euler Lagrange equation for absolute minimizers, i.e. local minima of $L^\infty$ functionals, typically the $L^\infty$ norm of the gradient. There has been a lot of work in more recent years to develop that theory using viscosity solutions by authors like Jensen \cite{jen1}, Barron-Jensen-Wang \cite{bjw}, Juutinen \cite{ju}, Crandall \cite{cr}. For the main results on the infinity Laplace equation, we refer the reader to the paper \cite{acj} and the references therein.
For results for equation (AE) especially in the $x$ dependent case, we also refer to the paper by the author \cite{soae} and the references therein, see also \cite{so2,so6}.
In particular we mention that equation (AE) has been studied in Carnot groups by Bieske-Capogna \cite{bica}, by Bieske \cite{bie} in the Grushin space, and by Wang \cite{wa} in the case of $C^2$ and homogeneous Hamiltonians with a Carnot Caratheodory structure.

The structure of the paper is as follows. In Section 2 we introduce the problem and give a motivating example. In Section 3 we introduce $C^1-$solutions of (AE) and show for them some important properties: monotonicity of the Hamiltonian on the hamiltonian dynamics, an equivalent definition and the fact that they are absolutely minimizing functions. In Section 4, we use $C^1-$solutions of (AE) as Lyapunov functions for nonlinear control systems and obtain local Lipschitz regularity of the minimum time function away from the singular set. In Section 5 we provide two new examples of explicit classical solutions of (AE) in two important cases of nonlinear control systems where the results of Section 4 apply.

\section{Control theory and the Aronsson equation}

As we mentioned in the introduction, throughout the paper we consider the controlled dynamical system (\ref{eqsystem})
where $\Omega\subset\R^n$ is open, $A$ is a nonempty,  compact subset of some metric space, $a_\cdot \in L^\infty((0,+\infty);A)$ and $f:\Omega\times A\to\R^n$ is a continuous function, continuously differentiable and uniformly Lipschitz continuous in the first group of variables, i.e.
$$|f(x^1,a)-f(x^2,a)|\leq L|x^1-x^2|\quad\mbox{for all }x^1,x^2\in\Omega,\;a\in A.$$
We suppose moreover that $f(x,A)$ is convex for every $x\in\Omega$ and that the system is symmetric, i.e. $-f(x,A)\subset f(x,A)$ for all $x\in \R^n$ and define the Hamiltonian
\begin{equation}\label{eqhamiltonian}
H(x,p)=\max_{a\in A}\{-f(x,a)\cdot p\}\in C(\Omega\times \R^n),
\end{equation}
so that $H\geq0$ and $H(x,-p)=H(x,p)$ by symmetry. Notice that $H$ is at least locally Lipschitz continuous, and $H(x,\cdot)$ is positively homogeneous of degree one by compactness of $A$.
We will also assume that $H$ is continuously differentiable on $\{(x,p):\in\Omega\times \R^{n}:H(x,p)>0\}$. 

The case we are mostly interested in the following sections is when
\begin{equation}\label{eqsigma}
f(x,a)=\sigma(x)a, \quad\sigma:\R^n\to M_{n\times m}
\end{equation}
where $M_{n\times m}$ is set of $n\times m$ matrices and $A=B_1(0)\subset\R^m$ is the closed unit ball. In this case $H(x,p)=|p\sigma(x)|$.

Given a smooth function $U\in C^1(\Omega)$ and $x_o\in \Omega\backslash{\mathcal H}$, where $\mathcal H$ is the singular set as in (\ref{eqsing}), we consider the hamiltonian dynamics
\begin{equation}\label{eqhd}
\left\{\begin{array}{ll}
\dot x_t=-H_p(x_t,\nabla U(x_t)),\\
 x_0=x_o\in\Omega,
\end{array}\right.
\end{equation}
where $H_p$ indicates the gradient of the Hamiltonian $H=H(x,p)$ with respect to the group of {\it adjoint} variables $p$.
\begin{rem}
When the Hamiltonian $H(x,\nabla U(x))$ is differentiable, notice that for $a_{x}\in A$ such that $-f(x,a_x)\cdot \nabla U(x)=H(x,\nabla U(x))$ we have that
$$-H_p(x,\nabla U(x))=f(x,a_x).$$
Therefore trajectories of (\ref{eqhd}) are indeed trajectories of the system (\ref{eqsystem}) and moreover (\ref{eqhd}) is a closed loop system of (\ref{eqsystem}) with feedback $a_x$. If in particular $f(x,a)$ is as in (\ref{eqsigma}), then, for $|p\sigma(x)|\neq0$, 
$$H(x,p)=|p\sigma(x)|, \quad H_p(x,p)=\sigma(x)\frac{^t\sigma(x) \;^tp}{H(x,p)}, \quad a_x=-\frac{^t\sigma(x)\nabla U(x)}{H(x,\nabla U(x))}\in B_1(0).$$
Therefore in this case the feedback control is at least continuous on $\Omega\backslash{\mathcal H}$ and the closed loop system always has a well defined local solution starting out on that set.
\end{rem}

We want to discuss when $H(x_t,\nabla U(x_t))$ is monotone on a trajectory $x_t$ of (\ref{eqhd}). If we can compute derivatives, then we need to discuss the sign of
$$ \frac d{dt}H(x_t,\nabla U(x_t))=\nabla(H(x_t,\nabla U(x_t)))\cdot \dot x_t=-\nabla(H(x_t,\nabla U(x_t)))\cdot H_p(x_t,\nabla U(x_t)).$$
Therefore a sufficient condition is that $U\in C^2(\Omega\backslash{\mathcal H})$ is a super or subsolution of the following pde
\begin{equation}\label{eqae}
-\nabla(H(x,\nabla U(x)))\cdot H_p(x,\nabla U(x))=0,\quad x\in \Omega\backslash{\mathcal H},
\end{equation}
which is named Aronsson equation in the literature. Notice that $H(x_t,\nabla U(x_t))$ is actually constant if $U$ is a classical solution of (\ref{eqae}). The above computation is correct only under the supposed regularity on $U$ and unfortunately if such regularity is not satisfied and we interpret super/subsolutions of (\ref{eqae}) as viscosity solutions this is no longer true in general, as the following example shows.
Notice that if $H$ is not differentiable at a point $(x_o,\nabla U(x_o))$ where $H(x_o,\nabla U(x_o))=0$, then $H_p(x_o,\nabla U(x_o))$ is multivalued, precisely the closed convex subgradient of the Lipschitz function $H(x_o,\cdot)$ computed at $
\nabla U(x_o)$and contains the origin by the symmetry of the system. Therefore the dynamics (\ref{eqhd}) has at least the constant solution also in this case. In some statements below it will be sometimes more convenient to look at (AE) for $H^2$ in order to gain regularity at points where $H$ vanishes.

\begin{exa}\label{exinfinity}
In the plane, suppose that $H^2(x,y,p_x,p_y)=(|p_x|^2+|p_y|^2)/2$ hence it is smooth and independent of the state variables. In this case (AE) becomes the well known infinity Laplace equation
$$-\Delta_\infty U(x)=-D^2U(x)\nabla U(x)\cdot\nabla U(x)=0.$$
It is easy to check that a viscosity solution of the equation is $u(x,y)=|x|^{4/3}-|y|^{4/3}$. The function $u\in C^{1,1/3}(\R^2)\backslash C^2$. Among solutions of the Hamiltonian dynamics $(\dot x_t,\dot y_t)=-\nabla U(x_t,y_t)$, we can find the following two trajectories
$$(x^{(1)}_t,y^{(1)}_t)=\left(\left(1-\frac89t\right)^{3/2},0\right),\quad(x^{(2)}_t,y^{(2)}_t)=\left(0,\left(1+\frac89t\right)^{3/2}\right),
$$
defined in a neighborhood of $t=0$. Clearly the Hamiltonian along the two trajectories is
$$H(\nabla U(x_t^{(1)},y_t^{(1)}))=\frac{2\sqrt{2}}3\sqrt{1-\frac89t},\quad H(\nabla U(x_t^{(2)},y_t^{(2)}))=\frac{2\sqrt{2}}3\sqrt{1+\frac89t},$$
it is strictly decreasing in the first case, strictly increasing in the second but it is never constant. Therefore the remark that we made at the beginning fails in this example. In the next section we are going to understand the reason.
\end{exa}

\section{Monotonicity of the Hamiltonian along the Hamiltonian dynamics}

Throughout this section, we consider a Hamiltonian not necessarily with the structure as in (\ref{eqhamiltonian}) but satisfying the following:
\begin{equation*}\tag{H1}\label{eqh1}
\begin{array}{c}
\;H:\Omega\times\R^n\to\R \hbox{ is continuous and }H(x,-p)=H(x,p),\\
H_p(x,p)\hbox{ exists and is continuous for all }
(x,p)\in \Omega\times\R^n\hbox{ if }H(x,p)>0.
\end{array}\end{equation*}
We will also refer to the following property:
\begin{equation}\tag{H2}\label{eqh2}
H(x,\cdot) \hbox{ is positively }r>0 \hbox{ homogeneous, for all }x\in\Omega.
\end{equation}
Given $U\in C^1(\Omega)$, the monotonicity of the Hamiltonian along trajectories of (\ref{eqhd}) is the object of this section. It is a consequence of the following known general result.
\begin{prop}\label{propmonotone}
Let $\Omega\subset\R^n$ be an open set and $F:\Omega\to\R^n$ be a continuous vector field. The following are equivalent:
\begin{itemize}
\item[(i)]{$V:\Omega\to\R$ is a continuous viscosity solution of $-F(x)\cdot\nabla V(x)\leq 0$ in $\Omega$.}
\item[(ii)]{The system $(V,F)$ is forward weakly increasing, i.e. for every $x_o\in\Omega$, there is a solution of the differential equation $\dot x_t=F(x_t)$, for $t\in [0,\gep)$, $x_0=x_o$ such that $V(x_s)\leq V(x_t)$ for $0\leq s\leq t$.}
\end{itemize}
Moreover the following are also equivalent
\begin{itemize}
\item[(iii)]{$V:\Omega\to\R$ is a continuous viscosity solution of $F(x)\cdot\nabla V(x)\geq 0$ in $\Omega$.}
\item[(iv)]{The system $(V,F)$ is backward weakly increasing, i.e. for every $x_o\in\Omega$, there is a solution of the differential equation $\dot x_t=F(x_t)$, for $t\in (-\gep,0]$, $x_0=x_o$ such that $V(x_s)\leq V(x_t)$ for $s\leq t\leq0$.}
\end{itemize}
\end{prop}
\begin{cor}\label{corclarke}
Let $\Omega\subset\R^n$ be an open set and $F:\Omega\to\R^n$ be a continuous vector field. The following are equivalent:
\begin{itemize}
\item[(i)]{$V:\Omega\to\R$ is a continuous viscosity solution of $-F(x)\cdot\nabla V(x)\leq 0$ and of $F(x)\cdot\nabla V(x)\geq 0$ in $\Omega$.}
\item[(ii)]{The system $(V,F)$ is weakly increasing, i.e. for every $x_o\in\Omega$, there is a solution of the differential equation $\dot x_t=F(x_t)$, for $t\in (-\gep,\gep)$, $x_0=x_o$ such that $V(x_s)\leq V(x_t)$ for $s\leq t$.}
\end{itemize}
\end{cor}
\begin{rem}
The proof of the previous statement can be found in \cite{clsw0}, see also \cite{clsw}.
When $F\in C^1$ another proof can be found in Proposition 5.18 of \cite{bcd} or can be deduced from the optimality principles in optimal control proved in  \cite{soopt}, when $F$ is locally Lipschitz continuous. In the case when $F$ is locally Lipschitz, the two differential inequalities in (i) of Corollary \ref{corclarke} turn out to be equivalent and of course there is also uniqueness of the trajectory of the dynamical system $\dot x=F(x)$, $x(0)=x_o$. When (ii) in the Corollary is satisfied by all trajectories of the dynamical system then the system is said to be strongly monotone. This occurs in particular if there is at most one trajectory, as when
$F$ is locally Lipschitz continuous. More general sufficient conditions for strong monotonicity can be found in \cite{clsw}, see also \cite{drw}.
\end{rem}

In view of the above result, we introduce the following definition.
\begin{defin}\label{defcone} Let $\Omega\subset\R^n$ be open and let $H:\Omega\times\R^n\to\R$ satisfying (\ref{eqh1}). We say that a function $U\in C^1(\Omega)$ is a $C^1-$supersolution (resp. subsolution) of the Aronsson equation (\ref{eqae}) in $\Omega$, if setting $V(x)=H(x,\nabla U(x))$ and $F(x)=-H_p(x,\nabla U(x))$ we have that $V$ is a viscosity subsolution (resp. supersolution) of
$-F(x)\cdot\nabla V(x)=0$ and a supersolution (resp. a subsolution) of $F(x)\cdot\nabla V(x)=0$.
\end{defin}
It is worth pointing out explicitely the consequence we have reached by Proposition \ref{propmonotone}.
\begin{cor}
Let $U\in C^1(\Omega)$ be a $C^1-$supersolution (resp, subsolution) of (\ref{eqae}). For $x_o\in\Omega\backslash{\mathcal H}$, then there is a trajectory $x_t$ of the Hamiltonian dynamics (\ref{eqhd}) such that $H(x_t,\nabla U(x_t))$ is nondecreasing (resp. nonincreasing).
\end{cor}
\begin{rem}
\begin{itemize}
\item{Notice that if $U$ is a $C^1-$solution of (\ref{eqae}) and the Hamiltonian dynamics (\ref{eqhd}) is either strongly decreasing and strongly increasing, as for instance if it has a unique solution for a given initial condition, then for all trajectories $x_t$ of (\ref{eqhd}), $H(x_t,\nabla U(x_t))$ is constant.
}
\item{In order to comment back to Example \ref{exinfinity}, notice that while $U(x,y)=|x|^{4/3}-|y|^{4/3}$ is a $C^1$ function, nevertheless, as easily checked, $V(x,y)=H^2(\nabla U(x,y))=16(|x|^{2/3}+|y|^{2/3})/9$ is only a viscosity subsolution but not a supersolution of
$$-\nabla V(x)\cdot (-H_p^2(\nabla U(x)))=0,$$
while it is a viscosity solution of $\nabla V(x)\cdot (-H_p^2(\nabla U(x)))=0$.
Then it turns out that the Hamiltonian is weakly increasing on the trajectories of the Hamiltonian dynamics.
Indeed there is another trajectory of the Hamiltonian dynamics such that $(x^{(3)}(0),y^{(3)}(0))=(1,0)=(x^{(1)}(0),y^{(1)}(0))$, namely
$$(x^{(3)}(t),y^{(3)}(t))=\left(\left(1-\frac89t\right)^{3/2},\left(\frac89t\right)^{3/2}\right)$$
along which the Hamiltonian is actually constant, until the trajectory is well defined. 
}
\item{It is clear by Example \ref{exinfinity} that while classical $C^2$ solutions of (\ref{eqae}) are $C^1-$solutions, continuous or even $C^1$ viscosity solutions in general are not. The definition of $C^1-$solution that we introduced is meant to preserve the monotonicity property of the Hamiltonian on the trajectories of the Hamiltonian dynamics.
}
\item{Observe that if $U$ is a $C^1-$solution, then $-U$ is a $C^1-$solution as well, since the Hamiltonian is unchanged and the vector field in the Hamiltonian dynamics becomes the opposite.
}\end{itemize}
\end{rem}

It may look unpleasant that Definition \ref{defcone} of solution of (\ref{eqae}) refers to a property that is not formulated directly for the function $U$. Therefore in the next statement we will reformulate the above definition. The property (ED) below will give an equivalent definition of a $C^1-$solution.
\begin{prop}
Let $U\in C^1(\Omega)$ and $H$ satisfying (\ref{eqh1}), (\ref{eqh2}). The following two statements are equivalent:
\begin{itemize}
\item[(ED)]{ for all $x_o\in\Omega\backslash{\mathcal H}$, there is a trajectory $x_t$ of the Hamiltonian dynamics (\ref{eqhd}), such that if $\varphi\in C^2([0,\gep))\cup C^2([(-\gep,0])$ is a test function and $U(x_t)-\varphi(t)$ has a minimum (respectively maximum) at $0$ and $\dot\varphi(0)=\frac{d}{dt}U(x_t)_{t=0}$, then we have that
$$-\ddot \varphi(0)\geq0\;(\hbox{resp. }\leq0).$$}
\item{$U$ is a $C^1-$supersolution (resp. subsolution) of (\ref{eqae}).}
\end{itemize}
In particular, if $H$ is $C^1$ at $\{(x,p):H(x,p)\neq0\}$, a $C^1-$supersolution (resp. subsolution) is a viscosity supersolution (resp. subsolution) of 
(\ref{eqae}).
\end{prop}
\begin{rem}
In the statement of (ED), when the hamiltonian vector field $F(x)=-H_p(x,\nabla U(x))$ is locally Lipschitz continuous, we may restrict the test functions to $\varphi\in C^2(-\gep,\gep)$.
\end{rem}
\begin{proof}
We only prove the statement for supersolutions, the other case being similar. Let $U\in C^1(\Omega)$.

Suppose first that (ED) holds true. Let $V(x)=H(x,\nabla U(x))$ and $\Phi\in C^1(\Omega)$ such that $V-\Phi$ has a maximum at $x_o$, $V(x_o)=\Phi(x_o)$. Therefore if $x_t$ is a solution of the hamiltonian dynamics (\ref{eqhd}) that satisfies (ED), we have that, by homogeneity of $H(x,\cdot)$ and for $F(x)=-H_p(x,\nabla U(x))$,
\begin{equation*}
r\Phi(x_t)\geq rV(x_t)=rH(x_t,\nabla U(x_t))=-\nabla U(x_t)\cdot F(x_t)=-\frac d{dt}U(x_t).
\end{equation*}
Thus integrating for small $t>0$ we get
\begin{equation}\label{eqmis}
\varphi(t):=U(x_o)-r\int_0^t\Phi(x_s)\;ds\leq U(x_t),
\end{equation}
and thus $U(x_t)-\varphi(t)$ has a minimum at $t=0$ on $[0,\gep)$ for $\gep$ small and $\dot\varphi(0)=-r\Phi(x_t)=\frac{d}{dt}U(x_t)|_{t=0}$. If instead $V-\Phi$ had a minimum at $x_o$, then integrating on $(t,0]$ for $t<0$ small enough, we would still obtain the same as in (\ref{eqmis}). 
By (ED), from (\ref{eqmis}) we get in both cases
$$0\geq\ddot\varphi(0)=-r\frac d{dt}\Phi(x_t)|_{t=0}=-r\nabla \Phi(x_o)\cdot F(x_o),$$
where $F(x)=-H_p(x,\nabla U(x))$. Therefore we conclude that $V$ is a viscosity subsolution of $-\nabla V\cdot F\leq 0$ (or a supersolution of $\nabla V\cdot F\geq 0$ when $V-\phi$ has a minimum st $x_o$). Finally by definition, $U$ is a $C^1-$supersolution of (\ref{eqae}).

Suppose now that $U$ is a $C^1-$supersolution of (\ref{eqae}). Then by Proposition \ref{propmonotone}, for all $x_o\in \Omega\backslash{\mathcal H}$, we can find a trajectory $x_t$ of the dynamics (\ref{eqhd}) such that $rV(x_t)=-\frac d{dt}U(x_t)$ is nondecreasing. Therefore $U(x_t)$ is a concave function of $t$. 
Let $\varphi\in C^2((-\gep,0])\cup C^2([0,\gep))$ be such that $U(x_t)-\varphi(t)$ has a minimum at $t=0$, $U(x_o)=\varphi(0)$ and $\frac d{dt}U(x_t)|_{t=0}=\dot\varphi(0)$. If we had $\ddot\varphi(0)>0$ then $\varphi$ would be strictly convex in its domain. Therefore for $t\neq0$ small enough, and in the domain of $\varphi$,
$$U(x_t)\geq\varphi(t)>\varphi(0)+\dot\varphi(0)t=U(x_o)+\frac d{dt}U(x_t)|_{t=0}t\geq U(x_t),
$$
by concavity of $U(x_t)$. This is a contradiction.

We prove the last statement on the fact that a $C^1-$solution is a viscosity solution. Therefore for a $C^1-$supersolution $U$ of (\ref{eqae}) let now $\Phi\in C^2(\Omega)$ be such that $U-\Phi$ has a minimum at $x_o$. By (ED), for a suitable solution $x_t$ of (\ref{eqhd}) we have that $U(x_t)-\varphi(t)$ has a minimum at $t=0$ if $\varphi(t)=\Phi(x_t)$, in particular $\dot\varphi(0)=\frac{d}{dt}U(x_t)_{t=0}$.
By (ED) and homogeneity of $H(x,\cdot)$,
$$\begin{array}{l}
0\leq -\ddot\varphi(0)=\frac d{dt}\nabla\Phi(x_t)\cdot H_p(x_t,\Phi(x_t))|_{t=0}=r\frac d{dt}H(x_t,\nabla \Phi(x_t))|_{t=0}
\\=-r\nabla(H(x_o,\nabla \Phi(x_o)))\cdot H_p(x_o,\nabla \Phi(x_o)).
\end{array}$$
Therefore $U$ is a viscosity supersolution of (\ref{eqae}). The case of subsolutions is similar and we skip it.
\end{proof}

We end this section by proving another important property of $C^1-$ solutions of (\ref{eqae}) that in the literature was the main motivation to the study of (AE).
\begin{thm}\label{teoam}
Let $\Omega\subset\R^n$ open and bounded, $H$ satisfying (\ref{eqh1}), and having the structure (\ref{eqhamiltonian}).
Let $U\in C^1(\Omega)\cap C({\overline\Omega})$ be a $C^1-$solution of (\ref{eqae}). For any function $W\in C({\overline\Omega})$ such that :
\begin{equation}\label{eqam}\left\{\begin{array}{ll}
H(x,\nabla W(x))\leq k\in\R,\quad& x\in\Omega,\\
W(x)=U(x),&x\in\partial\Omega
\end{array}\right.
\end{equation}
in the viscosity sense, then $H(x,\nabla U(x))\leq k$ in $\Omega$.
\end{thm}
\begin{rem}
When $D\subset\R^n$ is an open set and 
the property of a function $U\in C^1(D)$ in Theorem (\ref{teoam}) holds for all open subsets $\Omega\subset D$ then we say that $U$ is an {\it Absolutely minimizing function} in $D$ for the Hamiltonian $H$. This means that $U$ is a local minimizer of $\|H(\cdot,\nabla U(\cdot))\|_{L^\infty}$. It is well known that for the infinity Laplace equation, where we minimize the Lipschitz constant of $U$, it is equivalent to be a viscosity solution and an absolutely minimizing function. Such equivalence is also known for coercive Hamiltonians and for the norm of the horizontal gradient in some Carnot Caratheodory spaces. For more general Hamiltonians this equivalence is not known. Here we prove one implication at least for $C^1-$ solutions of (\ref{eqae}).
\end{rem}
\begin{proof}
Let $U,W$ be as in the statement and suppose for convenience that $H(x,\cdot)$ is positively 1-homogeneous.
We define $V(x)=H(x,\nabla U(x))\geq0$ and look at solutions $x_t$ of the Hamiltonian dynamics (\ref{eqhd}).
If $V(x_o)=0$, then clearly $V(x_o)\leq k$ and we have nothing left to show. If otherwise $V(x_o)>0$ since $U$ is a $C^1-$solution of (\ref{eqae}), we already know that 
we can construct a solution of (\ref{eqhd})  starting out at $x_o\in\Omega$ such that
$V(x_t)$ is nondecreasing for $t\geq0$ and nonincreasing for $t\leq0$ (by a concatenation of two  trajectories of (\ref{eqhd}) with monotone Hamiltonian). Since $\Omega$ is bounded, then the curve $x_t$ will not stay indefinitely in $\Omega$ because as we already observed
$$U(x_t)-U(x_o)\leq -\int_0^tV(x_s)\;ds\leq-t V(x_o),\quad \hbox{for }t\geq 0,$$
and
$$U(x_t)-U(x_o)\geq -t V(x_o),\quad \hbox{for }t\leq 0.$$
Hence $x_t$ will hit $\partial\Omega$ forward and backward in finite time. Let $t_1<0<t_2$ be such that $x_{t_1},x_{t_2}\in\partial\Omega$ and $x_t\in \Omega$ for $t\in(t_1,t_2)$. Therefore
\begin{equation}\label{eqab}
U(x_{t_2})+t_2V(x_o)\leq U(x_o)\leq U(x_{t_1})+t_1V(x_o)
\end{equation}
and then 
\begin{equation}\label{eqaa}
W(x_{t_1})-W(x_{t_2})=U(x_{t_1})-U(x_{t_2})\geq (t_2-t_{1})V(x_o).
\end{equation}
Now we use the differential inequality (\ref{eqam}) in the viscosity sense and the lower optimality principle in control theory as in \cite{soopt} for subsolutions of the Hamilton-Jacobi equation. Therefore since $x_t$ is a trajectory of the control system (\ref{eqsystem}) we have that for all $\gep>0$ and $t_1+\gep<t<t_2$, as $x_s\in\Omega$ for $s\in[t_1+\gep,t]$,
$$W(x_{t_1+\gep})\leq k(t-t_1-\gep)+W(x_{t}).$$
By letting $t\to t_2-$ and $\gep\to0+$ we conclude, by continuity of $W$ at the boundary of $\Omega$ and (\ref{eqaa}),
$$V(x_o)(t_2-t_1)\leq W(x_{t_1})-W(x_{t_2})\leq k(t_2-t_1)$$
which is what we want.
\end{proof}
\begin{rem}
Notice that in (\ref{eqab}) equalities hold if $V$ is constant on a given trajectory of (\ref{eqhd}) and we obtain that
$$\frac{U(x_o)-U(x_{t_1})}{t_1}=\frac{U(x_o)-U(x_{t_2})}{t_2}$$
and then
$$U(x_o)=\frac{t_2}{t_2-t_1}U(x_{t_1})-\frac{t_1}{t_2-t_1}U(x_{t_2}),$$
which is an implicit representation formula for $U$ through its boundary values, since the points $x_{t_1},x_{t_2}$ depend on the Hamiltonian dynamics (\ref{eqhd}) and $U$ itself.
\end{rem}

\section{Liapunov functions and (AE)}

In this section, we go back to the stucture (\ref{eqhamiltonian}) for $H$ and want to discuss the classical idea of control Lyapunov function. Let ${\mathcal T}\subset\R^n$ be a closed target set, we want to find $U:\R^n\to[0,+\infty)$ at least lower semicontinuous and such that: $U(x)=0$ if and only if $x\in{\mathcal T}$ and such that for all $x\in\R^n\backslash{\mathcal T}$ there exists a control $a_\cdot\in L^\infty(0,+\infty)$ and $t_x\leq+\infty$ such that the corresponding trajectory of (\ref{eqsystem}) satisfies:
$$U(x_t)\mbox{ is nonincreasing and }U(x_t)\to0,\quad\hbox{as }t\to t_x.$$
Classical necessary and sufficient conditions lead to look for strict supersolutions of the Hamilton Jacobi equation, namely to find $U$ such that
\begin{equation}\label{eqlyap}
H(x,\nabla U(x))\geq l(x),
\end{equation}
with $l:\R^n\to[0,+\infty)$ continuous and such that $l(x)=0$ if and only if $x\in{\mathcal T}$. The case ${\mathcal T}=\{0\}$ is already quite interesting for the theory.

Here we will apply the results of the previous section and plan consider Lyapunov functions built as follows. We analyse the existence of $U\in C^1(\Omega\backslash({\mathcal T}\cap{\mathcal H}))\cap {C(\overline{\Omega\backslash{\mathcal T}})}$ such that $U$ is a $C^1-$supersolution of (AE), i.e. satisfies
\begin{equation}\label{eqaei}
-\nabla(H(x,\nabla U(x))\cdot H_p(x,\nabla U(x)))\geq0\quad x\in\Omega\backslash({\mathcal T}\cap{\mathcal H}).\end{equation}
\begin{rem}
To study (\ref{eqaei}) in the case when $H$ is as in (\ref{eqhamiltonian}) and $f$ as in (\ref{eqsigma}), it is sometimes more convenient to write it for the Hamiltonian squared $H^2(x,\nabla U(x))=|\nabla U(x)\sigma(x)|^2$. Thus 
$$\begin{array}{ll}
-\nabla(H^2(x,\nabla U(x))\cdot (H^2)_p(x,\nabla U(x))=-4\;^tD(\nabla U\sigma(x))\;^t(\nabla U(x)\sigma(x))\cdot \left(\sigma(x)\;^t(\nabla U(x)\sigma(x))\right)\\
\quad=-4S^*\;^t(\nabla U(x)\sigma(x))\cdot \;^t(\nabla U(x)\sigma(x)),
\end{array}$$
where we indicated 
$$S=\;^t\sigma(x)^tD(\nabla U\sigma(x))=\;^t\sigma(x)D^2U(x)\sigma(x)+\left(D\sigma_j\sigma_i(x)\cdot \nabla U(x)\right)_{i,j=1,\dots,m},$$
$\sigma_j$, $j=1,\dots,k$ are the columns of $\sigma$, and $S^*=(S+\;^tS)/2$. Therefore a special sufficient condition for $U$ to satisfy (\ref{eqaei}) is that $S^*$ is negative semidefinite, which means that $U$ is $\sigma-$concave with respect to the family of vector fields $\sigma_j$, in the sense of Bardi-Dragoni \cite{badr}. We recall that the matrix $S$ also appears in \cite{so3} to study second order controllability conditions for symmetric control systems.
\end{rem}
Define the minimum time function for system (\ref{eqsystem}) as
$$T(x)=\inf_{a\in L^\infty(0,+\infty)}t_x(a),$$
where $t_x(a)=\inf\{t\geq0:x_t\in{\mathcal T},\;x_t \mbox{ solution of }(\ref{eqsystem})\}\leq+\infty$.
We prove the following result, recall that ${\mathcal H}=\{x:H(x,\nabla U(x))=0\}$ is the singular set.
\begin{prop}\label{propfeed}
Let $\Omega\subset\R^n$ be open and ${\mathcal T}\subset\Omega$ a closed target. Let $H$ have the structure (\ref{eqhamiltonian}).
Assume that $U\in C(\overline{\Omega\backslash{\mathcal T}})\cap C^1(\Omega\backslash({\mathcal T}\cap{\mathcal H}))$ is nonnegative and a $C^1-$solution of (\ref{eqaei}) in $\Omega\backslash({\mathcal T}\cap{\mathcal H})$ and that $U(x)=0$ for $x\in{\mathcal T}$, $U(x)=M$ for $x\in\partial\Omega$ and $U(x)\in (0,M)$ for $x\in\Omega\backslash{\mathcal T}$ and some $M>0$. For any $x_o\in\Omega\backslash({\mathcal T}\cup{\mathcal H})$ there exists a solution of the closed loop system (\ref{eqhd}) such that
\begin{itemize}
\item[(i)] {$H(x_t,\nabla U(x_t))$ is a nondecreasing function of $t$;
}
\item[(ii)]{$U(x_t)$ is a strictly decreasing function of $t$
}
\item[(iii)]{The trajectory $(x_t)_{t\geq0}$ reaches the target in finite time and the minimum time function for system (\ref{eqsystem}) satisfies the estimate
\begin{equation}\label{eqmte}
T(x_o)\leq \frac {U(x_o)}{H(x_o,\nabla U(x_o))}.
\end{equation}
}
\end{itemize}
\end{prop}
\begin{proof}
The thesis (i) follows from the results of the previous section since $U$ is a supersolution of (AE).
Let $x_o$ be a point where $H(x_o,\nabla U(x_o))>0$. By homogeneity of the Hamiltonian we get, for $t\geq0$
$$0<H(x_o,\nabla U(x_o))\leq H(x_t,\nabla U(x_t))=\nabla U(x_t)\cdot H_p(x_t,\nabla U(x_t))=-\frac d{dt}U(x_t)$$
and (ii) follows. Integrating now the last inequality we obtain
$$0\leq U(x_t)\leq U(x_o)-H(x_o,\nabla U(x_o))t$$
and thus the solution of (\ref{eqhd}) reaches the target before time 
\begin{equation}\label{eqtime}
\bar t=\frac{U(x_o)}{H(x_o,\nabla U(x_o))}.
\end{equation}
Therefore (\ref{eqmte}) follows by definition.
\end{proof}
The estimate (\ref{eqmte}) can be used to obtain local regularity of the minimum time function. The proof of regularity now follows a more standard path although under weaker assumptions than usual literature and will allow us to obtain a new regularity result.
We emphasize that nothing in the next statement is assumed on the structure of the vectogram $f(x,A)$ when $x\in{\mathcal T}$. In particular the target need not be even small time locally attainable.
\begin{thm}\label{thmregularity}
Let $\Omega\subset\R^n$ be open and ${\mathcal T}\subset\Omega$ a closed target. 
Assume that $U\in C(\overline{\Omega\backslash{\mathcal T}})\cap C^1(\Omega\backslash({\mathcal T}\cap{\mathcal H}))$ is nonnegative and $C^1-$solution of (\ref{eqaei}) in $\Omega\backslash({\mathcal T}\cap{\mathcal H})$ and that $U(x)=0$ for $x\in{\mathcal T}$, $U(x)=M$ for $x\in\partial\Omega$ and $U(x)\in (0,M)$ for $x\in\Omega\backslash{\mathcal T}$ and some $M>0$. Let $d(x)=\mbox{dist}(x,{\mathcal T})$ be the distance function from the target.
Suppose that $U$ satisfies the following: for all $\gep>0$ there are $\delta,c>0$ such that
\begin{equation}\label{eqexcond}
U(x)\leq c\; d(x),\quad \mbox{if }H(x,\nabla U(x))\geq\gep,\;d(x)<\delta.
\end{equation}
Then the minimum time function $T$ for system (\ref{eqsystem}) to reach the target is finite and locally Lipschitz continuous in $\Omega\backslash({\mathcal T}\cup{\mathcal H})$.
\end{thm}
\begin{proof}
Let $x_o\in\Omega$, $x_o\notin({\mathcal T}\cup{\mathcal H})$ and $r,\gep>0$ be such that $H(x,\nabla U(x))\geq\gep$, for all $x\in B_{r}(x_o)$. The parameter $r$ will be small enough to be decided later. We apply the assumption (\ref{eqexcond}) and find $\delta,c>0$ correspondingly. The fact that $T$ is finite in $B_r(x_o)$, for $r$ sufficiently small, follows from Proposition \ref{propfeed}.

Take $x^1,x^2\in B_r(x_o)$ and suppose that $x^1_t,x^2_t$ are the trajectories solutions of (\ref{eqsystem}) corresponding to the initial conditions $x_0=x^1,x^2$ respectively. To fix the ideas we may suppose that $T(x^2)\leq T(x^1)<+\infty$ and for any $\rho\in(0,1]$ we choose a control $a^\rho$ and time $t_2=t_{x^2}(a^\gep)\leq T(x^2)+\rho$ such that $d(x_{t_2})=0$.
Note that by (\ref{eqmte}), $t_2\leq \frac{U(x^2)}{\gep}+\rho\leq M_\gep$, for all $x^2\in B_r(x_o)$.
Moreover by the Gronwall inequality for system (\ref{eqsystem}) and since $d(x_{t_2})=0$,
$$d(x^1_{t_2})\leq|x^1_{t_2}-x^2_{t_2}|\leq|x^1-x^2|e^{Lt_2}\leq|x^1-x^2|e^{LM_\gep}$$
and the right hand side is smaller than $\delta$ if $r$ is small enough. 
Now we can estimate, by the dynamic programming principle and by (\ref{eqmte}), (\ref{eqexcond}),
$$0\leq T(x^1)-T(x^2)\leq (t_2+T(x^1_{t_2}))-t_2+\rho\leq \frac{U(x^1_{t_2})}\gep+\rho \leq\frac c\gep d(x^1_{t_2})+\rho\leq \frac{ce^{LM_\gep}}\gep|x^1-x^2|+\rho.$$
As $\rho\to0+$, the result follows.
\end{proof}
The extra estimate (\ref{eqexcond}) is crucial in the sought regularity of the minimum time function but contrary to the existing literature is only asked in a possibly proper subset of a neighborhood of the target. We will show in the examples of the next section how it may follow from (AE) as well. In order to achieve small time local attainability of the target, one needs in addition that the system can evade from $\mathcal H$.
\begin{cor}
In addition to the assumptions of Theorem \ref{thmregularity} suppose that ${\mathcal H}$ is a manifold of codimension at least one and that for all $x_o\in{\mathcal H}\cap(\Omega\backslash{\mathcal T})$ we have $f(x_o,A)\not\subset T_{x_o}({\mathcal H})$, the tangent space of $\mathcal H$ at $x_o$.
Then for any $x_o\in \Omega\backslash{\mathcal T}$ we can reach the target in finite time.
\end{cor}
\begin{proof}
By following the vector field $f(x_o,a)\notin T_{x_o}({\mathcal H})$, we immediately exit the singular set.
\end{proof}

\section{Some smooth explicit solutions of the Aronsson equation}

In this section we show two examples of well known nonlinear systems where we can find an explicit smooth solution of (AE) and then apply Theorem \ref{thmregularity} to obtain local Lipschitz regularity of the minimum time function. Our system will be in the form (\ref{eqhamiltonian}), (\ref{eqsigma}) and ${\mathcal T}=\{0\}$.

\subsection{H\"ormander-like vector fields.}
 
We consider the case where $x=(x_h,x_v)\in\R^{m+1}$ and
\begin{equation}\label{eqhormander}
\sigma(x)=\left(\begin{array}{cc}
I_m\\^t(Bx_h)\end{array}\right),\end{equation}
where $I_m$ is the $m\times m$ identity matrix and $B$ is not singular, $^tB=-B=B^{-1}$ is also $m\times m$. In particular $m$ is an even number and $|Bx_h|=|x_h|$.
It is known that the corresponding symmetric control system is globally controllable to the origin and that its minimum time function is locally $1/2-$H\"older continuous. We want to prove higher regularity except on its singular set.

We consider the two functions
\begin{equation}\label{eqgauge}
u(x)=|x_h|^4+4x_v^2,\quad U(x)=(u(x))^{1/4},
\end{equation}
and want to show that $U$ is a solution of (AE) for $H^2$ in $\R^{m+1}\backslash\{0\}$. $U$ is a so called gauge function for the family of vector fields.
We easily check that, after denoting $A(x)=\sigma(x)\;^t\sigma(x)$, 
$$\begin{array}{c}
\nabla u(x)=(4|x_h|^2x_h,8x_v),\quad
A(x)\;^t\nabla u(x)=\left(\begin{array}{cc}
4|x_h|^2x_h+8x_vBx_h\\8x_v|x_h|^2\end{array}\right),\\
H^2(x,\nabla u(x))=|\nabla U(x)\sigma(x)|^2=A(x)\;^t\nabla U(x)\cdot \;^t\nabla U(x)=16|x_h|^6+64x_v^2|Bx_h|^2=16|x_h|^2u(x),\\
H(x,\nabla U(x))=\frac{|x_h|}{U(x)}.
\end{array}$$
Notice in particular that $H(x,\nabla U(x))=0$ if and only if $x_h=0$ and thus
the singular set $\{x:H(x,\nabla U(x))=0\}$ contains the target and is a smooth manifold, being the $x_v$ axis.
As a consequence of the last displayed equation we have
$$ U(x)\leq\frac{|x_h|}\gep\leq\frac{|x|}\gep,\quad\hbox{in }H(x,\nabla U(x))\geq\gep,$$
which is an information that we need to apply Theorem \ref{thmregularity}.
Finally, if $x\neq0$,
$$\begin{array}{l}-\nabla(H^2(x,\nabla U(x)))\cdot (H^2)_p(x,\nabla U(x))
=-2\left(\frac{(x_h,0)}{U^2(x)}-\frac{|x_h|^2}{U^3(x)}\nabla U(x)\right)\cdot A(x)\;^t\nabla U(x)\\
=-\frac{2}{U^3(x)}\left(4U(x)\frac{|x_h|^4}{4U^3(x)}-|x_h|^2\frac{|x_h|^2}{U^2(x)}\right)=0.
\end{array}$$
Therefore $U$ is even a classical $C^2$ solution of (AE) for Hamiltonian $H^2$ in $\R^{m+1}\backslash\{0\}$ and then $H$ is constant along the trajectories of the closed loop system (\ref{eqhd}). Hence, by Theorem \ref{thmregularity}, the system (\ref{eqsystem}) is controllable in finite time to the origin from 
$$\{x:H(x,\nabla U(x))>0\}=\R^{m+1}\backslash\{(0,x_v):x_v\in\R\}$$
and the corresponding minimum time function is locally Lipschitz continuous on that set.
Notice that, for $\gep<1$, $\{x:H(x,\nabla U(x))\geq\gep\}=\{x:4x_v^2\leq(1/\gep^4-1)|x_h|^4\}$. Also the last Corollary applies.
\begin{prop}
Consider the symmetric control system
\begin{equation}\label{eqssystem}
\left\{\begin{array}{ll}
\dot x_t=\sigma(x_t)a_t,&\quad t>0,\\
x_o\in\R^n,
\end{array}\right.
\end{equation}
where $\sigma$ is given in (\ref{eqhormander}). Then the gauge function (\ref{eqgauge}) is a solution of the Aronsson equation (\ref{eqae}) for $H^2$ in $\R^{m+1}\backslash\{0\}$, it is an absolutely minimizing function for the corresponding $L^\infty$ norm of the subelliptic gradient and the minimum time function to reach the origin is locally Lipschitz continuous in $\{x=(x_h,x_v)\in\R^{m+1}:x_h\neq0\}$. The system is small time locally controllable and there is a continuous feedback leading the system to the target outside the singular set.
\end{prop}

\subsection{Grushin vector fields.}

We consider the system where $x=(x_h,x_v)\in \R^{m+1}$ and 
\begin{equation}\label{eqgrushin}
\sigma(x)=\left(\begin{array}{cc}
I_m\quad &0_m\\0&^tx_h\end{array}\right),
\end{equation}
where $\sigma(x)$ is $(m+1)\times 2m$ matrix.
Also in this case it is known that the corresponding symmetric control system is globally controllable to the origin and that its minimum time function is locally $1/2-$H\"older continuous.
We consider $u,\;U$ as before in (\ref{eqgauge}) want to show that $U$ is a solution of (AE)  in $\R^{m+1}\backslash\{0\}$.
In this case we can check that,
$$A(x)\;^t\nabla u(x)=\left(\begin{array}{cc}
4|x_h|^2x_h\\8x_v|x_h|^2\end{array}\right),\quad H^2(x,\nabla u(x))=16|x_h|^2u(x),
\quad H(x,\nabla U(x))=\frac{|x_h|}{U(x)},$$
and again we have, for $\gep>0$,
$$ U(x)\leq\frac{|x_h|}\gep\leq\frac{|x|}\gep,\quad\hbox{in }H(x,\nabla U(x))\geq\gep.$$
Finally, if $x\neq0$,
$$\begin{array}{l}-\nabla(H^2(x,\nabla U(x)))\cdot (H^2)_p(x,\nabla U(x))
=-\frac{2}{U^3(x)}\left(U(x)(x_h,0)-|x_h|^2\nabla U(x)\right)\cdot A(x)\;^t\nabla U(x)\\
=-\frac{2}{U^3(x)}\left(4U(x)\frac{|x_h|^4}{4U^3(x)}-|x_h|^2\frac{|x_h|^2}{U^2(x)}\right)=0.
\end{array}$$
Therefore $U$ is a solution of (AE) for Hamiltonian $H^2$ and hence the system (\ref{eqsystem}) is controllable in finite time to the origin from $\{x:H(x,\nabla U(x))>0\}$ and we prove the following result.
\begin{prop}
Consider the symmetric control system (\ref{eqssystem})
where $\sigma$ is given in (\ref{eqgrushin}). Then the gauge function (\ref{eqgauge}) is a solution of (AE) for $H^2$ in $\R^{m+1}\backslash\{0\}$, it is an absolutely minimizing function for the corresponding $L^\infty$ norm of the subelliptic gradient and the minimum time function to reach the origin is locally Lipschitz continuous in $\{x=(x_h,x_v)\in\R^{m+1}:x_h\neq0\}$.
\end{prop}

\end{document}